\documentclass[11pt,bezier]{article}
\usepackage{amsmath,amssymb,amsfonts,euscript,graphicx,fncychap,setspace,inputenc,fontenc,sidecap,float,epstopdf,tikz,animate,color}
\usetikzlibrary{positioning}
\usetikzlibrary{arrows}
\usetikzlibrary{decorations.markings}

\usepackage[margin=.75in]{geometry}

\title{{\bf Noetherian Rings Whose Annihilating-Ideal Graphs Have finite Genus}}
\author{{{\bf ${{\rm { {\bf F. ~Aliniaeifard }}}}$, ~${{\rm {\bf M.~ Behboodi}}}$}}\thanks
{ The research of the second author was  supported in part by a grant from IPM (No. 90160034).} ~{\bf and}
~${{\rm {\bf Y.~ Li}}}$\thanks
{ Corresponding author: Yuanlin Li, Department of Mathematics, Brock University, St. Catharines, Ontario,
Canada L2S 3A1. Fax:(905) 378-5713; E-mail address: yli@brocku.ca}
\thanks{ The research of the first and the third authors was supported in part by a Discovery Grant from the Natural Sciences and Engineering Research
Council of Canada.}
\thanks{The research of third author was also supported in part by the National Natural Science Foundation of China (No. 11271250).}\\
\\
\footnotesize{${^{\rm }}$}\vspace{-1mm}
}
 \def\Ann{{\rm Ann }}

\def\be{\begin{enumerate}}
\def\ee{\end{enumerate}}

\newtheorem{ttheo}{Theorem}[section]
\newtheorem{ccoro}[ttheo]{Corollary}
\newtheorem{llem}[ttheo]{Lemma}
\newtheorem{eexam}[ttheo]{Example}
\newtheorem{rrem}[ttheo]{Remark}
\newtheorem{ppro}[ttheo]{Proposition}

\newenvironment{pproof}{\noindent{\bf Proof. }}{}

\date{}
 \begin{document}
  \maketitle
\begin{abstract}
{ Let $R$ be a commutative ring and ${\Bbb{A}}(R)$ be the set of
ideals with non-zero annihilators. The annihilating-ideal graph of
$R$ is defined as the graph ${\Bbb{AG}}(R)$ with  vertex set
${\Bbb{A}}(R)^*={\Bbb{A}}\setminus\{(0)\}$ such that two distinct
vertices $I$ and $J$ are adjacent if and only if $IJ=(0)$.  We
characterize commutative Noetherian
   rings $R$  whose annihilating-ideal graphs have finite genus $\gamma(\Bbb{AG}(R))$. It
   is shown that if $R$ is a Noetherian ring such that  $0<\gamma(\Bbb{AG}(R))<\infty$,
   then  $R$ has only finitely many ideals.

  }
\end{abstract}

  \section{\large\bf Introduction}
  Throughout this paper, $R$
  denotes a commutative  ring with $1\neq 0$. Let $X$ be
 a subset of $R$. The {\it annihilator} of $X$
is the  ideal Ann$(X)=\{a\in R~:~aX=0\}$. We denote by $|Y|$ the cardinality of $Y$ and let
$Y^*=Y\setminus\{0\}$.  Let $V$ be a vector space over the field
$\Bbb{F}$. We use the notation ${{\rm v.dim}}_{\Bbb{F}}(V)$ to
denote the dimension of $V$ over the field $\Bbb{F}$.

   Let $S_{k}$ denote the sphere with $k$ handles where $k$ is a non-negative integer, that
is, $S_{k}$ is an oriented surface of genus $k$. The {\it genus}
of a graph $G$, denoted $\gamma(G)$, is the minimal integer $n$
such that the graph can be embedded in $S_{n}$. For details on the
notion of embedding a graph in a surface, see, e.g., \cite[Chapter
6]{white}. Intuitively, $G$ is embedded in a surface if it can be
drawn in the surface so that its edges intersect only at their
common vertices.  An
infinite graph $G$ is said to have infinite genus ($\gamma (G)=\infty$) if,
for every natural number $n$, there exists a finite subgraph
$G_{n}$ of $G$ such that $\gamma(G_{n})=n $. We note here that if
$H$ is a subgraph of a graph $G$, then $\gamma(H) \leq \gamma(G)$.
Let $K_n$ denote the complete graph on $n$ vertices; that is,
$K_n$ has vertex set $V$ with $|V|=n$ and
$a-b$ is an edge for every
distinct pair of vertices $a,b\in V$. Let $K_{m,n}$ denote the complete bipartite
graph; that is, $K_{m,n}$ has vertex set $V$ consisting of the
disjoint union of two subsets, $V_{1}$ and $V_2$, such that
$|V_{1}|= m$ and $|V_{2}|=n$, and
$a- b$ is an edge if and only if
$a \in V_{1}$ and $b\in V_{2}$. We may sometimes use
$K_{|V_{1}|,|V_{2}|}$  denote the complete bipartite graph with
vertex sets $V_{1}$ and $V_{2}$. Note that $K_{m,n} = K_{n,m}$. It
is well known that

${\rm (1.1)}~~~~~~~~~~~~~~~~~~~~~~~~~~~~~~~~~~~~~~~~~~~~~~~~~\gamma(K_{n})=\lceil \frac{(n-3)(n-4)}{12}\rceil ~~~{\rm for~ all}~ n \geq 3, ~ {\rm and} $\\

${\rm (1.2)}~~~~~~~~~~~~~~~~~~~~~~~~~~~~~~~~~~~~~~~~~~\gamma(K_{m,n})= \lceil \frac{(n-2)(m-2)}{4} \rceil ~~~{\rm for~all}~n \geq 2~ {\rm and}~ m \geq 2$\\
(see \cite{ring-you} and \cite{ring1}, respectively). For a graph
$G$, the degree of a vertex $\it{v}$ of $G$ is the number of edges
of $G$ incident with $\it{deg(v)}$.

Let $R$ be a ring.    We call an
 ideal $I$ of $R$ an {\it annihilating-ideal} if there exists a
 non-zero ideal $J$ of $R$ such that $IJ=(0)$.    We denote  by $\Bbb{A}(R)$ the set of
  all annihilating-ideals of $R$,   and  for an ideal $J$ of $R$, we denote by
 $\Bbb{I}(J)$ for the set $\{I~:~I~{\rm be~an~ideal~of}~R~{\rm and}~I\subseteq J\}$.
 Also,  by the {\it annihilating-ideal
   graph} $\Bbb{AG}(R)$ of $R$ we mean the graph with vertices
   $\Bbb{A}(R)^{*} = \Bbb{A}(R) \setminus \{(0)\}$ such that
   there is an (undirected) edge between vertices $I$ and $J$ if and
   only if $I\neq J$ and $IJ=(0)$. Thus, $\Bbb{AG}(R)$ is the empty
   graph if and only if $R$ is an integral domain. The notion of annihilating-ideal graph was first
introduced and systematically studied by Behboodi and Rakeei in
\cite{beh-rak1, beh-rak2}. Recently it has received a great deal
of attention from several authors, for instance,
\cite{aab,aan,AB2012} (see also \cite{vis}, in which the notion of
{\it``graph of zero-divisor ideals"} is investigated). Also, the
{\it zero-divisor graph} of $R$, denoted by $\Gamma(R)$,  is an
undirected graph with vertex set $Z(R)^*=Z(R)\setminus\{0\}$
such that two distinct
   vertices $x,y$ are adjacent if and only if $xy=0$ (where $Z(R)$ is the set of all zero divisors of $R$). The interplay between the ring theoretic properties of $R$ and the graph theoretic properties of $\Gamma(R)$,
begun in \cite{and-liv}. Several authors recently investigated the genus of a zero-divisor graph
   (for instance, see \cite{smi,wang,wick1,wick2}). In
   particular in \cite[Theorem 2]{wick2}, it was shown that for any positive integer $g$,
    there are finitely many finite commutative rings whose zero-divisor graphs have genus $g$. In \cite{AB2012}, the first and second authors investigated the genus of annihilating-ideal graphs. They showed that if $R$ is an Artinian ring with  $\gamma(\Bbb{AG}(R)) < \infty$, then either $R$ has only finitely many
ideals or $(R,\mathfrak{m})$ is a Gorenstein ring $(R,\mathfrak{m})$ with
${\rm v.dim}_{R/{\mathfrak{m}}}{\mathfrak{m}}/{\mathfrak{m}}^{2}\leq 2$.

    In this paper,  we continue the investigation of genus of annihilating-ideal graphs.  We first show that Noetherian rings $R$ (whose all nonzero proper ideals have nonzero annihilators) with $0<\gamma(\Bbb{AG}(R))<\infty$ have only finitely many ideals. Then we  characterize such Noetherian rings whose annihilating-ideal graphs have finite genus (including genus zero).

\section{\large\bf Preliminary}

We  list a few preliminary results which are needed to prove our main results.  The following
useful remark   will be used frequently in the sequel.

\begin{rrem}\label{2.2} {\rm It is well known that if $V$ is a vector space over an
infinite field $\Bbb{F}$, then $V$ can not be the union of
finitely many proper subspaces (see for example
\cite[p.283]{hal}).}
\end{rrem}

A local Artinian principal ideal ring is called a {\it special principal ideal
ring}
  and it has only finitely many ideals, each of which is a power of the maximal ideal.

\begin{llem}{\rm \cite[Lemma 2.3]{AB2012}}\label{2.3}
 Let $(R,\mathfrak{m})$ be a local ring with ${\mathfrak{m}}^t=(0)$. If for a positive integer
$n$, ${\it {\rm
v.dim}_{R/{\mathfrak{m}}}}({\mathfrak{m}}^{n}/{\mathfrak{m}}^{n+1})
=1$ and ${\mathfrak{m}}^{n}$ is a finitely generated $R$-module,
 then $\Bbb{I}({\mathfrak{m}}^{n})=\{{\mathfrak{m}}^i~:~n\leq i\leq t\}$. Moreover,
 if $n=1$, then $R$ is a  special principal ideal
ring.
 \end{llem}

\begin{llem}\label{n.1}
Let $(R,\mathfrak{m})$ be a local Artinian ring. If $I\in \Bbb{I}({\mathfrak{m}}^{n-1})\setminus \Bbb{I}({\mathfrak{m}}^{n})$ for some positive integer $n$, is a nonzero principal ideal, then $|\Bbb{I}(I)|=|\Bbb{I}(I\cap {\mathfrak{m}}^{n})|+1$.
\end{llem}

\begin{proof}
Since $I\in \Bbb{I}({\mathfrak{m}}^{n-1})\setminus \Bbb{I}({\mathfrak{m}}^{n})$ is a nonzero principal ideal, there exists  $x\in {\mathfrak{m}}^{n-1}\setminus {\mathfrak{m}}^{n}$ such that $I=Rx$. Let $J\in \Bbb{I}(Rx)$ such that $J\neq I$ and let $y\in J$. Thus $y=rx$ for some $r\in R$. If $r\not \in \mathfrak{m}$, then $r$ is an invertible element and so $Ry=Rx$, yielding a contradiction. Thus we have $r\in \mathfrak{m}$, so $y=rx\in \mathfrak{m}^{n}$. Therefore, $J \in \Bbb{I}(\mathfrak{m}^{n})$. Hence $|\Bbb{I}(I)|=|\Bbb{I}(I\cap {\mathfrak{m}}^{n})|+1$. \hfill $\square$
\end{proof}

\begin{llem}\label{n.2}
Let $(R,\mathfrak{m})$ be a Gorenstein ring. If $I$ is a principal ideal such that $|\Bbb{I}(I)|=3$, then $\mathfrak{m}^{2}\subseteq {\rm Ann(I)}$.
\end{llem}

\begin{proof}
Since $I$ is a principal ideal, $I=Rx$ for some $x\in R$. Since $Rx\cong R/{\rm Ann}(x)$ and $Rx$ has only one
 nonzero proper $R$-submodule,  ${\mathfrak{m}}/{\rm{Ann}}(x)$ is the only  nonzero proper ideal of $R/{\rm Ann}(x)$.
If ${\mathfrak{m}}^{2}\nsubseteq {\rm Ann}(x)$, then ${\rm
Ann}(x)+{\mathfrak{m}}^{2}={\mathfrak{m}}$, and  by Nakayama's
lemma (see \cite[(4.22)]{La91}), ${\rm Ann}(x)={\mathfrak{m}}$, yielding a contradiction. Thus
${\mathfrak{m}}^{2}\subseteq {\rm Ann}(x)={\rm Ann}(I)$.\hfill $\square$
\end{proof}

\begin{llem}\label{2.6}
Let $(R,\mathfrak{m})$ be an Artinian Gorenstein ring with $|R/\mathfrak{m}|=\infty$ such that $\mathfrak{m}^{t+1}=(0) (t\geq 5)$ ,\\ ${\rm v.dim}_{R/{\mathfrak{m}}}{\mathfrak{m}}^{t-s}/{\mathfrak{m}}^{t-(s-1)}\leq2$ and ${\rm v.dim}_{R/{\mathfrak{m}}}{\mathfrak{m}}^{t-(s-1)}/{\mathfrak{m}}^{t-(s-2)}= 1$ (where $s=t-1$ or $t-2$). If $\gamma(\Bbb{AG}(R))<\infty$, then ${\rm v.dim}_{R/{\mathfrak{m}}}{\mathfrak{m}}^{t-s}/{\mathfrak{m}}^{t-(s-1)}=1$.
\end{llem}

\begin{pproof} Let $k=t-s$. Suppose on the  contrary that ${\rm v.dim}_{R/{\mathfrak{m}}}{\mathfrak{m}}^{k}/{\mathfrak{m}}^{k+1}= 2$. Let $x_1\in {\mathfrak{m}}^{k}\setminus {\mathfrak{m}}^{k+1}$. Assume that ${\mathfrak{m}}^{t-2}\nsubseteq Rx_1$. By Lemma~\ref{n.1}, $|\Bbb{I}(Rx_1)|=|\Bbb{I}({\mathfrak{m}}^{k+1}\cap Rx_1)|+1$. Since ${\mathfrak{m}}^{t-2}\nsubseteq Rx_1$, either ${\mathfrak{m}}^{k+1}\cap Rx_1={\mathfrak{m}}^{t}$ or ${\mathfrak{m}}^{k+1}\cap Rx_1={\mathfrak{m}}^{t-1}$. We conclude that  $|\Bbb{I}(Rx_1)|$ is either $3$ or $4$. If $|\Bbb{I}(Rx_1)|=3$, then by Lemma~\ref{n.2},
${\mathfrak{m}}^{2}\subseteq {\rm Ann}(x_1)$, and so $\mathfrak{m}^{t-2}\subseteq {\rm Ann}(x_1)$.

Assume that $|\Bbb{I}(Rx_1)|=4$.  If $\mathfrak{m}^3\subseteq {\rm Ann}(x_{1})$, then $\mathfrak{m}^{t-2}\subseteq {\rm Ann}(x_{1})$. Suppose that $\mathfrak{m}^3\nsubseteq {\rm Ann}(x_{1})$. Since $Rx_1\cong R/{\rm Ann}(x_1)$ and $Rx_1$ has only two nonzero proper ideals, ${\mathfrak{m}}$ and ${\rm Ann}(x_1)+{\mathfrak{m}}^{2}$ are the only nonzero proper ideals of $R/{\rm Ann}(x_1)$. Therefore, ${\rm Ann}(x_1)+{\mathfrak{m}}^{2}={\rm Ann}(x_1)+{\mathfrak{m}}^{3}$, and so $({\rm Ann}(x_1)+{\mathfrak{m}}^{2}){\mathfrak{m}}^{t-4}=({\rm Ann}(x_1)+{\mathfrak{m}}^{3}){\mathfrak{m}}^{t-4}$. Thus
${\rm Ann}(x_1)\mathfrak{m}^{t-4}+{\mathfrak{m}}^{t-2}={\rm Ann}(x_1){\mathfrak{m}}^{t-4}+{\mathfrak{m}}^{t-1}\subseteq {\rm Ann}(x_1)$. Therefore, ${\mathfrak{m}}^{t-2}\subseteq {\rm Ann}(x_1)$. So, in both cases, ${\mathfrak{m}}^{t-2}\subseteq {\rm Ann}(x_1)$. Thus if ${\mathfrak{m}}^{t-2}\nsubseteq  Rx_1$, then ${\mathfrak{m}}^{t-2}\subseteq {\rm Ann}(x_1)$.

Assume that ${\mathfrak{m}}^{t-2}\subseteq Rx_1$. Note that ${\mathfrak{m}}^{k}x_1\cong {\mathfrak{m}}^{k}/(Ann(x_1)\cap {\mathfrak{m}}^{k})$. Since ${\mathfrak{m}}^{k}x_1\subseteq {\mathfrak{m}}^{k+1}$ and by Lemma \ref{2.3} $|\Bbb{I}({\mathfrak{m}}^{k+1})|<\infty$, we have ${\rm v.dim}_{R/\mathfrak{m}} (({\rm Ann}(x_1)\cap {\mathfrak{m}}^{k})+{\mathfrak{m}}^{k+1})/{\mathfrak{m}}^{k+1}=1$, so ${\rm Ann}(x_1)\cap ({\mathfrak{m}}^{k}\setminus {\mathfrak{m}}^{k+1})\neq (0)$. Let $x_2\in {\rm Ann}(x_1)\cap ({\mathfrak{m}}^{k}\setminus {\mathfrak{m}}^{k+1})$. Then ${\mathfrak{m}}^{t-2}\subseteq {\rm Ann}(x_2)$ since ${\mathfrak{m}}^{t-2}\subseteq Rx_1$. Therefore, ${\mathfrak{m}}^{t-2}\subseteq {\rm Ann}(x_1)$ or ${\mathfrak{m}}^{t-2}\subseteq {\rm Ann}(x_2)$.\\

 Let $x_{2i-1}\in {\mathfrak{m}}^{k}\setminus {\mathfrak{m}}^{k+1}$ $(i\geq 2)$ such that $\{x_{2i-1}+{\mathfrak{m}}^{k+1}, x_2+{\mathfrak{m}}^{k+1}\}$ and  $\{x_{2i-1}+{\mathfrak{m}}^{k+1}, x_{2j-1}+{\mathfrak{m}}^{k+1}\}$ (for each $j=1, 2,\cdots, i-1)$ are  bases for ${\mathfrak{m}}^{k}/{\mathfrak{m}}^{k+1}$.

 We claim that ${\mathfrak{m}}^{t-2}\subseteq Rx_{2i-1}$. Suppose on the contrary that ${\mathfrak{m}}^{t-2}\nsubseteq Rx_{2i-1}$. Therefore, by Lemma~\ref{n.1}, $|\Bbb{I}(Rx_{2i-1})|=|\Bbb{I}({\mathfrak{m}}^{k+1}\cap Rx_{2i-1})|+1$. Since ${\mathfrak{m}}^{t-2}\nsubseteq Rx_{2i-1}$, ${\mathfrak{m}}^{k+1}\cap Rx_{2i-1}={\mathfrak{m}}^{t}$ or ${\mathfrak{m}}^{k+1}\cap Rx_{2i-1}={\mathfrak{m}}^{t-1}$. We conclude that $|\Bbb{I}(Rx_{2i-1})|$ is  either $3$ or $4$. If $|\Bbb{I}(Rx_{2i-1})|=3$, then by Lemma~\ref{n.2},
${\mathfrak{m}}^{2}\subseteq {\rm Ann}(x_{2i-1})$, and so ${\mathfrak{m}}^{t-2}\subseteq {\rm Ann}(x_{2i-1})$.

 Assume that $|\Bbb{I}(Rx_{2i-1})|=4$. If $\mathfrak{m}^3\subseteq {\rm Ann}(x_{2i-1})$, then $\mathfrak{m}^{t-2}\subseteq {\rm Ann}(x_{2i-1})$. Suppose that $\mathfrak{m}^3\nsubseteq {\rm Ann}(x_{2i-1})$. Since $Rx_{2i-1}\cong R/{\rm Ann}(x_{2i-1})$ and $Rx_{2i-1}$ has only two nonzero proper ideals, ${\mathfrak{m}}$ and ${\rm Ann}(x_{2i-1})+{\mathfrak{m}}^{2}$ are the only nonzero proper ideals of $R/{\rm Ann}(x_{2i-1})$. Therefore, ${\rm Ann}(x_{2i-1})+{\mathfrak{m}}^{2}={\rm Ann}(x_{2i-1})+{\mathfrak{m}}^{3}$, and so  $({\rm Ann}(x_{2i-1})+{\mathfrak{m}}^{2}){\mathfrak{m}}^{t-4}=({\rm Ann}(x_{2i-1})+{\mathfrak{m}}^{3}){\mathfrak{m}}^{t-4}$. Thus ${\rm Ann}(x_{2i-1})\mathfrak{m}^{t-4}+{\mathfrak{m}}^{t-2}={\rm Ann}(x_{2i-1}){\mathfrak{m}}^{t-4}+{\mathfrak{m}}^{t-1}\subseteq {\rm Ann}(x_{2i-1})$. Therefore, ${\mathfrak{m}}^{t-2}\subseteq {\rm Ann}(x_{2i-1})$.

Since either ${\mathfrak{m}}^{t-2}\subseteq {\rm Ann}(x_1)$ or ${\mathfrak{m}}^{t-2}\subseteq {\rm Ann}(x_2)$, we conclude that either $(0)={\mathfrak{m}}^{t-2} (Rx_1+Rx_{2i-1})={\mathfrak{m}}^{t-2}{\mathfrak{m}}^{2}={\mathfrak{m}}^{t}\not=(0)$ or $(0)={\mathfrak{m}}^{t-2} (Rx_2+Rx_{2i-1})={\mathfrak{m}}^{t-2}{\mathfrak{m}}^{2}={\mathfrak{m}}^{t}\not=(0)$, yielding a contradiction.
 Thus ${\mathfrak{m}}^{t-2}\subseteq Rx_{2i-1}$.

Since ${\mathfrak{m}}^{2}x_{2i-1}\subseteq {\mathfrak{m}}^{k+1}$ and $|\Bbb{I}({\mathfrak{m}}^{k+1})|<\infty$, $|\Bbb{I}({\mathfrak{m}}^{k}x_{2i-1})|<\infty$. Note that ${\mathfrak{m}}^{k}x_{2i-1}\cong {\mathfrak{m}}^{k}/(\Ann(x_{2i-1})\cap {\mathfrak{m}}^{k})$. Thus there are finitely many ideals between  ${\rm Ann}(x_{2i-1})\cap {\mathfrak{m}}^{k}$ and ${\mathfrak{m}}^{k}$. So ${\rm v.dim}_{R/{\mathfrak{m}}} ({\rm Ann}(x_{2i-1})\cap {\mathfrak{m}}^{k}+{\mathfrak{m}}^{k+1})/{\mathfrak{m}}^{k+1}=1$, hence ${\rm Ann}(x_{2i-1})\cap ({\mathfrak{m}}^{k}\setminus {\mathfrak{m}}^{k+1})\neq (0)$. \\

Therefore, we can find $x_{2i}\in {\rm Ann}(x_{2i-1})\cap ({\mathfrak{m}}^{k}\setminus {\mathfrak{m}}^{k+1})$ and so ${\mathfrak{m}}^{t-2}\subseteq {\rm Ann}(x_{2i})$. If $Rx_{2i}=Rx_{2j}$ for some $j=1,2,\cdots, i-1$, then $Rx_{2i}(Rx_{2i-1}+Rx_{2j-1})=0$,  so $Rx_{2i}{\mathfrak{m}}^{k}=0$. Since ${\mathfrak{m}}^{t}{\mathfrak{m}}^{k}={\mathfrak{m}}^{t-1}{\mathfrak{m}}^{k}=(Rx_{2i-1}){\mathfrak{m}}^{k}=(0)$, $K_{|\Bbb{I}({\mathfrak{m}}^{k})|,3}$ is a subgraph of $\Bbb{AG}(R)$, so by Formula~(1.2) $\gamma(\Gamma(R))=\infty$, yielding a contradiction. Since $Rx_{2i}\in \Bbb{I}({\rm Ann}({\mathfrak{m}}^{t-2}))$ for each $i$, we conclude that $|\Bbb{I}({\rm Ann}({\mathfrak{m}}^{t-2}))|=\infty$.  Since ${\mathfrak{m}}^{t}{\rm Ann}({\mathfrak{m}}^{t-2})={\mathfrak{m}}^{t-1}{\rm Ann}({\mathfrak{m}}^{t-2})={\mathfrak{m}}^{t-2}{\rm Ann}({\mathfrak{m}}^{t-2})=(0)$, $K_{|\Bbb{I}({\rm Ann}({\mathfrak{m}}^{t-2}))|,3}$ is a subgraph of $\Bbb{AG}(R)$,  so $\gamma(\Bbb{AG}(R))=\infty$, yielding a contradiction. Therefore,
 ${\rm v.dim}_{R/{\mathfrak{m}}}{\mathfrak{m}}^{k}/{\mathfrak{m}}^{k+1}= 1$.\hfill $\square$

\end{pproof}

\section{\large\bf Main Results}

We first show that local Noetherian rings whose annihilating-ideal graphs have positive genus are Artinian. Next we show that Artinian rings whose annihilating-ideal graphs have positive genus have finitely many ideals [Theorem \ref{2.7}]. Finally, we characterize Noetherian rings whose all nonzero proper ideals have nonzero annihilators and annihilating-ideal graphs have finite genus (including genus zero) [Theorem \ref{last}].

\begin{ppro}\label{nlocal}
Let $R$ be a local  ring with $0<\gamma(\Bbb{AG}(R))<\infty$.  Then $R$ is Artinian if and only if  $R$ is Noetherian.
\end{ppro}

\begin{pproof}
It suffices to show that if $(R,\mathfrak{m})$ is a local Noetherian ring with $ 0 < \gamma(\Bbb{AG}(R))<\infty$, then $(R,\mathfrak{m})$ is Artinian. By \cite[Corollary 3.6]{AB2012}, either $R$ is a Gorenstein ring or $R$ is an Artinian ring with only finitely many ideals. We may assume that $(R,\mathfrak{m})$ is a Gorenstein ring. We want to show that $(R,\mathfrak{m})$ is Artinian. If there exists a positive integer $n$ such that $\mathfrak{m}^{n}=\mathfrak{m}^{n+1}$, then by Nakayama's lemma, $\mathfrak{m}^{n}=(0)$. For every prime ideal $P$ of $R$, $\mathfrak{m}^{n}\subseteq P$. Thus $\mathfrak{m}=P$. Since every prime ideal is a maximal ideal, we conclude that $R$ is Artinian. So we may assume that for every positive integer $n$, $\mathfrak{m}^{n+1}\subsetneq \mathfrak{m}^n$.

If there exist $x,y\in \mathfrak{m}$ such that $|\Bbb{I}(Rx)|=3$, $|\Bbb{I}(Ry)|=3$, and $Rx\neq Ry$.  Then by Lemma \ref{n.2}, $\mathfrak{m}^2\subseteq {\rm Ann}(x) \cap {\rm Ann}(y)$. Since ${\rm Ann}(\mathfrak{m})\mathfrak{m}^2=(0), (Rx)\mathfrak{m}^2=(0)$ and $(Ry)\mathfrak{m}^2=(0)$, we have $K_{|\Bbb{I}(\mathfrak{m}^2)|,3}$ is a subgraph of $\Bbb{\gamma(\Bbb{AG}(R))}$. Note that $|\Bbb{I}(\mathfrak{m}^2)|=\infty$. Thus by formula (1.2), $\gamma(\Bbb{AG}(R))=\infty$, yielding a contradiction. Therefore, $R$ has at most one principal ideal $Rx$ such that $|\Bbb{I}(Rx)|=3$.

Let  $Rz$ be a principal ideal of $R$ other than $Rx$ and ${\rm Ann}(\mathfrak{m})$, Thus $|\Bbb{I}(Rz)|\geq 4$. If $|\Bbb{I}({\rm Ann}(z))|=\infty$, then $K_{|\Bbb{I}({\rm Ann}(z))|,3}$ is a subgraph of $\Bbb{AG}(R)$, and so by formula (1.2), $\gamma(\Bbb{AG}(R))=\infty$, yielding a contradiction. Thus we may assume that  $|\Bbb{I}({\rm Ann}(z))|<\infty$. We have the following cases:

{\bf Case 1:} $|\Bbb{I}(Rz)|<\infty$.  Then $Rz$ and ${\rm Ann}(z)$ are Artinian $R$-modules.  Since $Rz\cong R/{\rm Ann}(z)$,  by \cite[(1.20)]{La91} we conclude that $R$ is Artinian.

{\bf Case 2:} $ |\Bbb{I}(Rz)|=\infty$.  If  $|\Bbb{I}({\rm Ann}(z))|\geq 4$, Then $K_{|\Bbb{I}(Rz)|,3}$ is a subgraph of $\Bbb{AG}(R)$. Thus by formula (1.2), $\gamma(\Bbb{AG}(R))=\infty$, yielding a contradiction. We conclude that $|\Bbb{I}({\rm Ann}(z))|\leq 3$, and so ${\rm Ann}(z)\in \{{\rm Ann}(\mathfrak{m}), Rx\}$. Therefore, every principal ideal other than ${\rm Ann}(\mathfrak{m})$ and $Rx$ has degree at most 2, and it is adjacent to ${\rm Ann}(\mathfrak{m})$ or $Rx$. We can also conclude that every ideal other than ${\rm Ann}(\mathfrak{m})$ and $Rx$ has degree at most 2.  Thus $\Bbb{AG}(R)$ is isomorphic to a subgraph of Figure 1, and so $\gamma(\Bbb{AG}(R))=0$, yielding a contradiction.\hfill $\square$
\end{pproof}

\begin{center}
\begin{tikzpicture}[scale=0.5]
\tikzstyle{every node}=[draw, shape=circle, inner sep=2pt];
\node (v0) at (0,0)[draw, circle, fill][label=right:${{\rm Ann}(\mathfrak{m})}$]{};
\node (v1) at (4,3)[draw, circle, fill][label=right:${~~~~\ldots}$]{};
\node (v2) at (3,3)[draw, circle, fill][label=left:${}$]{};
\node (v3) at (2,3)[draw, circle, fill][label=right:${}$]{};
\node (v4) at (1,3)[draw, circle, fill][label=right:${}$]{};
\node (v5) at (-1,3)[draw, circle, fill][label=right:${}$]{};
\node (v6) at (-2,3)[draw, circle, fill][label=right:${}$]{};
\node (v7) at (-3,3)[draw, circle, fill][label=above right:${}$]{};
\node (v8) at (-4,3)[draw, circle, fill][label=left:${\ldots~~~~}$]{};
\node (v9) at (4,-3)[draw, circle, fill][label=right:${~~~~\ldots}$]{};
\node (v10) at (3,-3)[draw, circle, fill][label=above right:${}$]{};
\node (v11) at (2,-3)[draw, circle, fill][label=above right:${}$]{};
\node (v12) at (1,-3)[draw, circle, fill][label=above right:${}$]{};
\node (v13) at (-1,-3)[draw, circle, fill][label=above right:${}$]{};
\node (v14) at (-2,-3)[draw, circle, fill][label=above right:${}$]{};
\node (v15) at (-3,-3)[draw, circle, fill][label=above right:${}$]{};
\node (v16) at (-4,-3)[draw, circle, fill][label=left:${\ldots~~~~}$]{};
\node (v17) at (0,6)[draw, circle, fill][label=above:${Rx}$]{};

\draw (v0) -- (v1) (v2) -- (v0) --(v3)(v4) -- (v0)-- (v5) (v6)--(v0)--(v7)(v8)--(v0)--(v9)(v10) --(v0)--(v11)(v12)--(v0)-- (v13)(v14)--(v0)--(v15)(v16)--(v0)--(v17)--(v1)(v17)--(v2)(v17)--(v3)(v17)--(v4)(v17)--(v5)(v17)--(v6)(v17)--(v7)(v17)--(v8);
\end{tikzpicture}
\end{center}

\begin{center}
Figure 1
\end{center}

We now state our first main result.

\begin{ttheo}\label{2.7}

 Let $R$ be an Artinian  ring with $0<\gamma(\Bbb{AG}(R))<\infty$. Then $R$ has only finitely many ideals.

\end{ttheo}

\begin{proof}
By \cite[Theorem 2.7]{AB2012}, we have either $R$ has only  finitely many ideals or $(R,\mathfrak{m})$ is a Gorenstein ring with ${\rm v.dim}_{R/\mathfrak{m}}\mathfrak{m}/\mathfrak{m}^2\leq2$.   So we may assume that $(R,\mathfrak{m})$  is a Gorenstein ring with ${\rm v.dim}_{R/\mathfrak{m}}\mathfrak{m}/\mathfrak{m}^2\leq 2$.  If
$|R/{\mathfrak{m}}|<\infty$, then one can easily see that $R$ is a
finite ring,  so $R$ has only finitely many ideals. Thus we can
assume that $|R/{\mathfrak{m}}|=\infty$, $\gamma(\Bbb{AG}(R))=
g$ for an integer $g>0$, and ${\it {\rm v.dim}_{R/{\mathfrak{m}}}} ({\rm Ann}({\mathfrak{m}})) =1$.
 Since $R$ is an Artinian ring, there exists a positive integer $t$
 such that
 ${\mathfrak{m}}^{t+1}=(0)$ and ${\mathfrak{m}}^{t} \neq (0)$. Note that ${\mathfrak{m}}^{t}\subseteq {\rm Ann}(\mathfrak{m})$ and ${\rm dim}({\rm Ann}(\mathfrak{m}))=1$, thus ${\rm Ann}(\mathfrak{m})={\mathfrak{m}}^{t} $. Let $I$ be a minimal ideal of $R$. Then $I\mathfrak{m}=(0)$. Hence $I\subseteq {\rm Ann}(\mathfrak{m})$ and so $I=\mathfrak{m}^{t}$. Therefore, $\mathfrak{m}^{t}$ is the {\bf unique minimal ideal} of $R$. We now proceed the proof using the case by case analysis.

 {\bf Case 1:}  $t=1$, i.e., ${\mathfrak{m}}^{2}=(0)$. Then since $\mathfrak{m}=\mathfrak{m}^{t}$ is also the unique minimal ideal, $R$ has exactly two proper ideals as desired.

{\bf Case 2:} $t=2$, i.e., ${\mathfrak{m}}^{2}\neq (0)$ and
${\mathfrak{m}}^{3}=(0)$. We first  assume that ${\it {\rm v.dim}_{R/{\mathfrak{m}}}} {\mathfrak{m}}/{\mathfrak{m}}^{2}= 2$. Let $I\neq {\mathfrak{m}}^{2}$ be an ideal. If $I\neq {\mathfrak{m}}$, then
 since ${\mathfrak{m}}^{2} \subseteq I$ and ${\rm v.dim}_{R/\mathfrak{m}}I/\mathfrak{m}^{2}=1$, we conclude that $I=Rx$ for some $x \in \mathfrak{m}\setminus \mathfrak{m}^{2}$. Note that $Rx\cong R/ {\rm Ann}(x)$. Since  $|\Bbb{I}(Rx)^{*}|=2$ (by Lemma~\ref{n.1}), we conclude that there is exactly one ideal between ${\rm Ann}(x)$ and $\mathfrak{m}$. Therefore, ${\it {\rm v.dim}_{R/{\mathfrak{m}}}} {\rm Ann}(x)/{\mathfrak{m}}^{2}=1$. So,   by Lemma~\ref{n.1} $|\Bbb{I}({\rm Ann}(x))^{*}|=2$. We conclude that  every ideal except for $\mathfrak{m}^{2}$ has degree at most $2$, so $\Bbb{AG}(R)$ is a subgraph of Figure 2, hence $\gamma(\Bbb{AG}(R))=0$, yielding a contradiction. Therefore, ${\it {\rm v.dim}_{R/{\mathfrak{m}}}} {\mathfrak{m}}/{\mathfrak{m}}^{2}=1$. Thus by
 Lemma~\ref{2.3},  $R$ is a special principal ideal ring and has only two nonzero proper ideals ${\mathfrak{m}}$ and ${\mathfrak{m}}^{2}$. 

\begin{center}
\begin{tikzpicture}[scale=0.6]
\tikzstyle{every node}=[draw, shape=circle, inner sep=2pt];
\node (v0) at (0,0)[draw, circle, fill][label=right:${{\rm Ann}(\mathfrak{m})}$]{};
\node (v1) at (4,3)[draw, circle, fill][label=right:${~~~~\ldots}$]{};
\node (v2) at (3,3)[draw, circle, fill][label=left:${}$]{};
\node (v3) at (2,3)[draw, circle, fill][label=right:${}$]{};
\node (v4) at (1,3)[draw, circle, fill][label=right:${}$]{};
\node (v5) at (-1,3)[draw, circle, fill][label=right:${}$]{};
\node (v6) at (-2,3)[draw, circle, fill][label=right:${}$]{};
\node (v7) at (-3,3)[draw, circle, fill][label=above right:${}$]{};
\node (v8) at (-4,3)[draw, circle, fill][label=left:${\ldots~~~~}$]{};
\node (v9) at (4,-3)[draw, circle, fill][label=right:${~~~~\ldots}$]{};
\node (v10) at (3,-3)[draw, circle, fill][label=above right:${}$]{};
\node (v11) at (2,-3)[draw, circle, fill][label=above right:${}$]{};
\node (v12) at (1,-3)[draw, circle, fill][label=above right:${}$]{};
\node (v13) at (-1,-3)[draw, circle, fill][label=above right:${}$]{};
\node (v14) at (-2,-3)[draw, circle, fill][label=above right:${}$]{};
\node (v15) at (-3,-3)[draw, circle, fill][label=above right:${}$]{};
\node (v16) at (-4,-3)[draw, circle, fill][label=left:${\ldots~~~~}$]{};

\draw (v0) -- (v1)-- (v2) -- (v0) --(v3)--(v4) -- (v0)-- (v5) --(v6)--(v0)--(v7)--(v8)--(v0)--(v9)(v10) --(v0)--(v11)(v12)--(v0)-- (v13)(v14)--(v0)--(v15)(v16)--(v0);
\end{tikzpicture}
\end{center}

\begin{center}
{Figure 2}
\end{center}

{\bf Case 3:}   $t =3$, i.e., ${\mathfrak{m}}^{4}=(0)$. Recall that ${\rm v.dim}_{R/{\mathfrak{m}}} {\mathfrak{m}}/{\mathfrak{m}}^{2}\leq 2$. First assume that ${\rm v.dim}_{R/{\mathfrak{m}}} {\mathfrak{m}}/{\mathfrak{m}}^{2}=2$. If ${\rm v.dim}_{R/{\mathfrak{m}}} {\mathfrak{m}}^{2}/{\mathfrak{m}}^{3}\geq 2$, then since ${\mathfrak{m}}^{4}=(0)$, we conclude that $K_{|\Bbb{I}(\mathfrak{m}^{2})|-1}$ is a subgraph of $\Bbb{AG}(R)$. Note that $|\Bbb{I}(\mathfrak{m}^{2})|=\infty$. Thus by Formula (1.1), $\gamma(\Bbb{AG}(R))=\infty$, yielding a contradiction. Therefore,  ${\rm v.dim}_{R/{\mathfrak{m}}} {\mathfrak{m}}^{2}/{\mathfrak{m}}^{3}=1$, and so by Lemma~\ref{2.3}, $\Bbb{I}({\mathfrak{m}}^{2})^{*}=\{{\mathfrak{m}}^{2}, {\mathfrak{m}}^{3}\}$.

  We now claim that  $| \Bbb{I}({\rm Ann}({\mathfrak{m}}^{2}))|=\infty$. Let $x_1\in \mathfrak{m}\setminus \mathfrak{m}^{2}$.  Then by Lemma~\ref{n.1}, $|\Bbb{I}(Rx_1)|=|\Bbb{I}(\mathfrak{m}^{2}\cap Rx_1)|+1<\infty$.
	Thus there are finitely many ideals between ${\rm Ann}(x_{1})$ and ${\mathfrak{m}}$. So ${\rm v.dim} ({\rm Ann}(x_{1})+{\mathfrak{m}}^{2})/{\mathfrak{m}}^{2}=1$. Hence ${\rm Ann}(x_{1})\cap ({\mathfrak{m}}\setminus {\mathfrak{m}}^{2})\neq (0)$. Let $x_2\in {\rm Ann}(x_1)$. If  $\mathfrak{m}^{2}\nsubseteq Rx_1$, then by Lemma~\ref{n.1}, $|\Bbb{I}(Rx_1)|=|\Bbb{I}(m^{2}\cap Rx_1)|+1$. Thus $|\Bbb{I}(Rx_1)|=3$, and by Lemma~\ref{n.2}, $\mathfrak{m}^{2}\subseteq {\rm Ann}(x_1)$. Therefore,  ${\mathfrak{m}}^{2}\subseteq {\rm Ann}(x_1)$ or $\mathfrak{m}^{2}\subseteq Ann(x_2)$.
	 Let $x_{2i-1}\in {\mathfrak{m}}\setminus {\mathfrak{m}}^{2}$ $(i\geq 2)$ such that $\{x_{2i-1}+{\mathfrak{m}}^{2}, x_2+{\mathfrak{m}}^{2}\}$ and  $\{x_{2i-1}+{\mathfrak{m}}^{2}, x_{2j-1}+{\mathfrak{m}}^{2}\}$ for $j=1, 2,\cdots, i-1$ be a basis for ${\mathfrak{m}}/{\mathfrak{m}}^{2}$.
 If ${\mathfrak{m}}^{2}\nsubseteq Rx_{2i-1}$, then ${\mathfrak{m}}^{2}\subseteq {\rm Ann}(x_{2i-1})$. Since either ${\mathfrak{m}}^{2}\subseteq {\rm Ann}(x_1)$ or ${\mathfrak{m}}^{2}\subseteq{\rm Ann}(x_2)$, we conclude that either $(0)={\mathfrak{m}}^{2} (Rx_1+Rx_{2i-1})={\mathfrak{m}}^{2}{\mathfrak{m}}={\mathfrak{m}}^{3}$ or $(0)={\mathfrak{m}}^{2} (Rx_2+Rx_{2i-1})={\mathfrak{m}}^{2}{\mathfrak{m}}={\mathfrak{m}}^{3}$, yielding a contradiction. Thus ${\mathfrak{m}}^{2}\subseteq Rx_{2i-1}$. As in the above, ${\rm Ann}(x_{2i-1})\cap ({\mathfrak{m}}\setminus {\mathfrak{m}}^{2})\neq (0)$.
We can find $x_{2i}\in {\rm Ann}(x_{2i-1})\cap ({\mathfrak{m}}\setminus {\mathfrak{m}}^{2})$. Since ${\mathfrak{m}}^{2}\subseteq Rx_{2i-1}$, ${\mathfrak{m}}^{2}\subseteq {\rm Ann}(x_{2i})$. If $Rx_{2i}=Rx_{2j}$ for some $j=1,2,\cdots, i-1$, then $Rx_{2i}(Rx_{2i-1}+Rx_{2j-1})=0$, and so $Rx_{2i}{\mathfrak{m}}=0$, yielding a contradiction since ${\rm Ann}({\mathfrak{m}})={\mathfrak{m}}^{t}$. Since $Rx_{2i}\in \Bbb{I}({\rm Ann}({\mathfrak{m}}^{2}))$ for each $i$, so $| \Bbb{I}({\rm Ann}({\mathfrak{m}}^{2}))|=\infty$, as we claimed.  Next we characterize the degree of every ideal $\mathfrak{m}\neq I\not \in \Bbb{I}({\mathfrak{m}}^{2})$ in $\Bbb{AG}(R)$:

{\bf Type1:} ${\mathfrak{m}}^{2}\nsubseteq I$.  By  Lemma \ref{n.1},  $|\Bbb{I}(Ry)|=3$, thus it is easy to see that I is a principal ideal, so  $I=Ry$ for some $y\in \mathfrak{m}$. Also, by Lemma~\ref{n.2}, $\mathfrak{m}^{2}\subseteq {\rm Ann}(y)$. Since ${\rm v.dim}({\rm Ann}(y)/{\mathfrak{m}}^{2})=1$, there exists $z\in {\rm Ann}(y)\setminus {\mathfrak{m}}^{2}$ such that  ${\rm Ann}(y)=Rz+{\mathfrak{m}}^{2}$. If ${\mathfrak{m}}^{2}\subseteq Rz$, we conclude that ${\rm Ann}(y)=Rz$. So by Lemma~\ref{n.1}, $|\Bbb{I}({\rm Ann}(y))|=4$. Therefore,   $deg(I)=3$ and $I$ is adjacent to $\mathfrak{m}^2$ and $\mathfrak{m}^3$.
     If ${\mathfrak{m}}^{2}\nsubseteq Rz$, then $|\Bbb{I}(Rz)|=3$ and by Lemma~\ref{n.2}, ${\mathfrak{m}}^{2}Rz=(0)$. Since ${\rm Ann}(y)=Rz+{\mathfrak{m}}^{2}$ and ${\mathfrak{m}}^{2}Rz=(0)$, we conclude that  ${\rm Ann}(y)={\rm Ann}({\mathfrak{m}}^{2})$. Thus  $Ry{\rm Ann}({\mathfrak{m}}^{2})=(0)$, ${\mathfrak{m}}^{2}{\rm Ann}({\mathfrak{m}}^{2})=(0)$, and ${\mathfrak{m}}^{3}{\rm Ann}({\mathfrak{m}}^{2})=(0)$. Therefore, $K_{|\Bbb{I}({\rm Ann}({\mathfrak{m}}^{2}))|,3}$ is a subgraph of $\Bbb{AG}(R)$, and so by Formula~(1.2), $\gamma(\Bbb{AG}(R))=\infty$, yielding a contradiction.

{\bf Type 2:} ${\mathfrak{m}}^{2}\subseteq I$. Since ${\rm v.dim}_{R/\mathfrak{m}}(I/{\mathfrak{m}}^{2})=1$, there exists $z\in I\setminus {\mathfrak{m}}^{2}$ such that  $I=Rz+{\mathfrak{m}}^{2}$. If ${\mathfrak{m}}^{2}\subseteq Rz$, we conclude that $I=Rz$.  If $\mathfrak{m}^{2}\subseteq {\rm Ann}(z)$, then there exists exactly one ideal between ${\rm Ann}(z)$ and $R$. Since $Rz\cong R/{\rm Ann}(z)$ and by Lemma~\ref{n.1}, $|\Bbb{I}(Rz)|=4$, we conclude that there exist two ideals between ${\rm Ann}(z)$ and $R$, yielding a contradiction. Thus ${\mathfrak{m}}^{2}\nsubseteq {\rm Ann}(z)$.  Therefore ${\rm Ann}(z)$ is a ideal of type 1. So ${\rm Ann(z)}$ is a principal ideal and $|\Bbb{I}({\rm Ann}(z))|=3$, thus $deg(I)=2$. If ${\mathfrak{m}}^{2}\nsubseteq Rz$, then by Lemma \ref{n.1} and Lemma~\ref{n.2}, ${\mathfrak{m}}^{2}Rz=(0)$. Since $I=Rz+{\mathfrak{m}}^{2}$ and ${\mathfrak{m}}^{2}Rz=(0)$, we conclude that  $I={\rm Ann}({\mathfrak{m}}^{2})$. If $deg(I)\geq 3$, then $K_{|\Bbb{I}({\rm Ann}({\mathfrak{m}}^{2}))|,3}$ is a subgraph of $\Bbb{AG}(R)$. Since $|\Bbb{I}({\rm Ann}({\mathfrak{m}}^{2}))|=\infty$, by Formula~(1.2), $\gamma(\Bbb{AG}(R))=\infty$, yielding a contradiction.  Therefore, $deg(I)=2$ and $I$ is only adjacent to $\mathfrak{m}^2$ and $\mathfrak{m}^3$.

 Therefore, every ideal (except for ${\mathfrak{m}}^{2}$ and $\mathfrak{m}^{3}$) has degree at most $3$, and every ideal with degree 3 is adjacent to ${\mathfrak{m}}^{2}$ and $\mathfrak{m}^{3}$. We  conclude that $\gamma(\Bbb{AG}(R))$ is a subgraph of the graph in Figure 3 below, and so  $\gamma(\Bbb{AG}(R))=0$, yielding a contradiction.

\begin{center}
\begin{tikzpicture}[scale=0.5]
\tikzstyle{every node}=[draw, shape=circle, inner sep=2pt];
\node (v0) at (0,0)[draw, circle, fill][label=right:${{\rm Ann}(\mathfrak{m})}$]{};
\node (v1) at (4,3)[draw, circle, fill][label=right:${~~~~\ldots}$]{};
\node (v2) at (3,3)[draw, circle, fill][label=left:${}$]{};
\node (v3) at (2,3)[draw, circle, fill][label=right:${}$]{};
\node (v4) at (1,3)[draw, circle, fill][label=right:${}$]{};
\node (v5) at (-1,3)[draw, circle, fill][label=right:${}$]{};
\node (v6) at (-2,3)[draw, circle, fill][label=right:${}$]{};
\node (v7) at (-3,3)[draw, circle, fill][label=above right:${}$]{};
\node (v8) at (-4,3)[draw, circle, fill][label=left:${\ldots~~~~}$]{};
\node (v9) at (4,-3)[draw, circle, fill][label=right:${~~~~\ldots}$]{};
\node (v10) at (3,-3)[draw, circle, fill][label=above right:${}$]{};
\node (v11) at (2,-3)[draw, circle, fill][label=above right:${}$]{};
\node (v12) at (1,-3)[draw, circle, fill][label=above right:${}$]{};
\node (v13) at (-1,-3)[draw, circle, fill][label=above right:${}$]{};
\node (v14) at (-2,-3)[draw, circle, fill][label=above right:${}$]{};
\node (v15) at (-3,-3)[draw, circle, fill][label=above right:${}$]{};
\node (v16) at (-4,-3)[draw, circle, fill][label=left:${\ldots~~~~}$]{};
\node (v17) at (0,6)[draw, circle, fill][label=above:${\mathfrak{m}^3}$]{};

\draw (v0) -- (v1) (v2) -- (v0) --(v3)(v4) -- (v0)-- (v5) --(v6)--(v0)--(v7)--(v8)--(v0)--(v9)(v10) --(v0)--(v11)(v12)--(v0)-- (v13)(v14)--(v0)--(v15)(v16)--(v0)--(v17)--(v1)(v17)--(v2)(v17)--(v3)(v17)--(v4)(v17)--(v5)(v17)--(v6)(v17)--(v7)(v17)--(v8);
\end{tikzpicture}
\end{center}

\begin{center}
{Figure 3}
\end{center}


{\bf Case 4:}   $t \geq 4$. Since ${\mathfrak{m}}^{3}{\mathfrak{m}}^{t}={\mathfrak{m}}^{3}{\mathfrak{m}}^{t-1}={\mathfrak{m}}^{3}{\mathfrak{m}}^{t-2}=(0)$,
$K_{|\Bbb{I}({\mathfrak{m}}^{3})|-3,3}$ is a subgraph of
$\Bbb{AG}(R)$. So by Formula~(1.2),
$\lceil(|\Bbb{I}({\mathfrak{m}}^{3})|-5)/4\rceil\leq g$. Hence,
$|\Bbb{I}({\mathfrak{m}}^{3})| \leq 4g+5$. If ${\it {\rm
v.dim}_{R/{\mathfrak{m}}}}({\mathfrak{m}}^{t-1}/{\mathfrak{m}}^{t})\geq
2$, then Remark \ref{2.2} implies that
$|\Bbb{I}({\mathfrak{m}}^{t-1})|=\infty$.
 Since
${\mathfrak{m}}^{t-1}{\mathfrak{m}}^{t-1}=(0)$ and $t\geq 3$,
$K_{|\Bbb{I}({\mathfrak{m}}^{t-1})|-1}$ is a subgraph of
$\Bbb{AG}(R)$. Therefore, by Formula~(1.1),  $\gamma(\Bbb{AG}(R)) = \infty$, a
contradiction. Thus ${\it {\rm
v.dim}_{R/{\mathfrak{m}}}}({\mathfrak{m}}^{t-1}/{\mathfrak{m}}^{t})=
1$. Hence, by Lemma~\ref{2.3}, there exists
$x \in {\mathfrak{m}}^{t-1}$ such that ${\mathfrak{m}}^{t-1}= Rx$.

Now we prove that ${\rm v.dim}_{R/{\mathfrak{m}}}{\mathfrak{m}}/{\mathfrak{m}}^{2}= 1$. Suppose on the contrary that ${\rm {\rm
v.dim}}_{R/{\mathfrak{m}}}{\mathfrak{m}}/{\mathfrak{m}}^{2}=2$. By Lemma~\ref{n.2},
${\mathfrak{m}}^{2}\subseteq {\rm Ann}(x)$, and since
${\mathfrak{m}}/{\rm Ann}(x)$ is the only nonzero proper ideal of
$R/{\rm Ann}(x)$, ${\rm v.dim}_{R/{\mathfrak{m}}}{\rm
Ann}(x)/{\mathfrak{m}}^{2}=1$. Let
$\{y_{1}+{\mathfrak{m}}^{2},y+{\mathfrak{m}}^{2}\}$
be a basis for ${\mathfrak{m}}/{\mathfrak{m}}^{2}$ such that
$\{y_{1}+{\mathfrak{m}}^{2}\}$
is a basis for
${\rm Ann}(x)/{\mathfrak{m}}^{2}$. Since ${\mathfrak{m}}^{t}$ is the only minimal ideal of $R$, ${\mathfrak{m}}^{t}\subseteq Ry$ and $|\Bbb{I}(Ry)|\geq 3$.
If  $|\Bbb{I}(Ry)|=3$, then  by Lemma~\ref{n.2},
${\mathfrak{m}}^{2}\subseteq {\rm Ann}(y)$. So
${\mathfrak{m}}^{t-1}\subseteq {\rm Ann}(y)$. Hence,
${\mathfrak{m}}^{t-1}{\mathfrak{m}}=(0)$
(${\mathfrak{m}}^{t-1}(Ry)=(0)$ and
${\mathfrak{m}}^{t-1}(Ry_{1})=(0)$), a contradiction.
Therefore,
$|\Bbb{I}(Ry)|\geq 4$. We now divided the proof into two subcases according to whether or not $|\Bbb{I}(Ry)|=\infty$, and show that both subcases are impossible.

{\bf Subcase 4.1:} $4 \leq |\Bbb{I}(Ry)|<\infty$.  We now claim that
  ${\rm v.dim}_{R/{\mathfrak{m}}}{\mathfrak{m}}^{2}/{\mathfrak{m}}^{3}=1$.
Suppose that ${\rm v.dim}_{R/{\mathfrak{m}}}{\mathfrak{m}}^{2}/{\mathfrak{m}}^{3}=l\geq 3$.
 Since ${\mathfrak{m}}^{2}y\cong {\mathfrak{m}}^{2}/\Ann(y)\cap {\mathfrak{m}}^{2}$ and $|\Bbb{I}({\mathfrak{m}}^{2}y)|\leq |\Bbb{I}({\mathfrak{m}}^{3})|<\infty$, we conclude that ${\rm v.dim}_{R/\mathfrak{m}} (\Ann(y)\cap {\mathfrak{m}}^{2}+{\mathfrak{m}}^{3}/{\mathfrak{m}}^{3})=l-1$, and so $|\Bbb{I}(Ann(y))|=\infty$. Since $K_{|\Bbb{I}(\Ann(y))|, |\Bbb{I}(Ry)|}$ is a subgraph of $\Bbb{AG}(R)$, by Formula (1.2), $\gamma(\Bbb{AG}(R))=\infty$,  yielding a contradiction. So we  have
 ${\rm v.dim}_{R/{\mathfrak{m}}}{\mathfrak{m}}^{2}/{\mathfrak{m}}^{3}\leq 2$.

Assume that $t=4$. Suppose that  ${\rm v.dim}_{R/{\mathfrak{m}}}{\mathfrak{m}}^{2}/{\mathfrak{m}}^{3}= 2$.  Recall that ${\rm v.dim}_{R/{\mathfrak{m}}}{\mathfrak{m}}/{\mathfrak{m}}^2= 2$. If there exists $w\in \mathfrak{m}\setminus \mathfrak{m}^3$ such that $|\Bbb{I}(Rw)|=3$, then by Lemma \ref{n.2}, $\mathfrak{m}^2\subseteq {\rm Ann(w)}$. Therefore, $(Rw)\mathfrak{m}^2=\mathfrak{m}^3\mathfrak{m}^2=\mathfrak{m}^4 \mathfrak{m}^2=(0)$, and so $K_{|\Bbb{I}(\mathfrak{m}^2)|,3}$ is a subgraph of $\Bbb{AG}(R)$. Thus by Formula (1.2), $\gamma(\Bbb{AG}(R))=\infty$, yielding a contradiction. So for every $w\in \mathfrak{m}\setminus \mathfrak{m}^3$ we have  $|\Bbb{I}(Rw)|\geq 4$.

 Let $s\in \mathfrak{m}^2\setminus \mathfrak{m}^3$. Then by Lemma \ref{n.2}, we have $\mathfrak{m}^3\subseteq Rs$. Since
$\Bbb{I}(Rw)=\Bbb{I}(Rw\cap \mathfrak{m}^2)\cup \{Rw\}$,
$|\Bbb{I}(Rw)|\geq 4$.  We conclude that every ideal except $\mathfrak{m}^4$ contains $\mathfrak{m}^3$ and so
every ideal except $\mathfrak{m}^4$ contains $\mathfrak{m}^3$. Let $\{r_1+\mathfrak{m}^2,r_2+\mathfrak{m}^2\}$ be a basis for $\mathfrak{m}/\mathfrak{m}^2$. Then since $\mathfrak{m}^3\subseteq {\rm Ann}(r_1)\cap {\rm Ann}(r_2)$, we have $\mathfrak{m}^3\mathfrak{m}=(0)$, yielding a contradiction. Therefore,  ${\rm v.dim}_{R/{\mathfrak{m}}}{\mathfrak{m}}^{2}/{\mathfrak{m}}^{3}= 1$.

Let $v\in \mathfrak{m}\setminus \mathfrak{m}^2$.  Then by Lemma \ref{n.2}, $|\Bbb{I}(Rv)|= 3$, $4$ or $5$.
 If $|\Bbb{I}(Rv)|=5$, then since $Rv\cong R/{\rm Ann(v)}$, we must have  three ideals between ${\rm Ann(v)}$ and $R$. But since by Lemma \ref{n.2}, we have ${\rm Ann}(v)+\mathfrak{m}^3\subseteq {\rm Ann}(v)+\mathfrak{m}^2$. Thus there only exist two ideals between ${\rm Ann(v)}$ and $R$, yielding a contradiction. Therefore, for every  $v\in \mathfrak{m}\setminus \mathfrak{m}^2$, $|\Bbb{I}(Rv)|=3$ or  $4$. If $|\Bbb{I}(Rv)|=3$, then by Lemma \ref{n.2}, $\mathfrak{m}^2\subseteq Rv$.  $|\Bbb{I}(Rv)|=4$, then since $Rv\cong R/{\rm Ann}(v)$, there exist only two ideals between ${\rm Ann}(v)$ and $R$, we conclude that  $\mathfrak{m}^3\subseteq {\rm Ann}(v)$. Again, same as above $\mathfrak{m}^3\mathfrak{m}=(0)$, yielding a contradiction.

So we assume that $t\geq 5$. Then by Lemma \ref{2.6}, ${\rm v.dim}_{R/{\mathfrak{m}}}{\mathfrak{m}}^{2}/{\mathfrak{m}}^{3}=1$. Since ${\rm v.dim}_{R/{\mathfrak{m}}}{\mathfrak{m}}/{\mathfrak{m}}^{2}\leq 2$, again by  Lemma \ref{2.6}, ${\rm v.dim}_{R/{\mathfrak{m}}}{\mathfrak{m}}/{\mathfrak{m}}^{2}=1$, yielding a contradiction.

{\bf Subcase 4.2:} $|\Bbb{I}(Ry)|=\infty$.
 Suppose  that ${\rm v.dim}_{R/{\mathfrak{m}}}{\mathfrak{m}}^{2}/{\mathfrak{m}}^{3}\geq
 2$.  If $|\Bbb{I}({\rm Ann}(y))|\geq 4$, then $K_{|\Bbb{I}(Ry)|,3}$ is a subgraph of
 $\Bbb{AG}(R)$ and so by Formula~(1.2),
  $\gamma(\Bbb{AG}{R})=\infty$, a contradiction. We may assume that $|\Bbb{I}({\rm Ann}(y))|=3$.
   Since ${\mathfrak{m}}^{2}y\cong {\mathfrak{m}}^{2}/({\rm Ann}(y)\cap {\mathfrak{m}}^{2})$, there exist only
finitely many
 ideals between ${\rm Ann}(y)$ and ${\mathfrak{m}}^{2}$. Therefore, ${\rm Ann}(y)\nsubseteq {\mathfrak{m}}^{3}$ and so  there exists $z\in {\rm Ann}(y)\setminus
 {\mathfrak{m}}^{3}$.  Since $|\Bbb{I}({\rm Ann}(y))|=3$, $Rz={\rm Ann}(y)$. Thus by Lemma~\ref{n.2}
${\mathfrak{m}}^{2}\subseteq {\rm Ann}(z)={\rm Ann}({\rm
Ann}(y))$. Since ${\mathfrak{m}}^{2}{\rm Ann}(y)=(0)$,
${\mathfrak{m}}^{2}{\mathfrak{m}}^{t-1}=(0)$, and
 ${\mathfrak{m}}^{t}{\mathfrak{m}}^{2}=(0)$, $K_{|\Bbb{I}({\mathfrak{m}}^{2})|,3}$ is a subgraph of
 $\Bbb{AG}(R)$. So by Formula~(1.2),
  $\gamma(\Bbb{AG}(R))=\infty$, a contradiction. Therefore,
 ${\rm v.dim}_{R/{\mathfrak{m}}}{\mathfrak{m}}^{2}/{\mathfrak{m}}^{3}=1$ and by Lemma \ref{2.3},
 $|\Bbb{I}({\mathfrak{m}}^{2})|<\infty$. Also, by Lemma~\ref{n.1}, $|\Bbb{I}(Ry)|<\infty$, yielding a contradiction.\\

 Therefore, we have either $R$ is a finite ring or ${\rm v.dim}_{R/{\mathfrak{m}}}{\mathfrak{m}}/{\mathfrak{m}}^{2}=1$. In the latter case,  by Lemma~\ref{2.3} $R$ has only finitely many ideals. The proof is complete. \hfill $\square$
\end{proof}

\begin{ccoro}
Let $R$ be a Noetherian ring such that all non-trivial ideals of $R$ are vertices of $\Bbb{AG}(R)$. If $0<\gamma(\Bbb{AG}(R))<\infty$, then $R$ is an Artinian ring with only finitely many ideals.
\end{ccoro}

\begin{pproof}
By \cite[Theorem 3.5]{AB2012}, $R$ is a Gorenstein ring or $R$ is
an Artinian ring with only finitely many ideals. So we may assume that $R$ is a Gorenstein ring. By Proposition \ref{nlocal}, $R$ is Artinian, and thus by Theorem \ref{2.7}, $R$ has only finitely many ideals.\hfill $\square$
\end{pproof}
\\

We are now ready to state our  second main result. As a consequence of Proposition \ref{nlocal} and  Theorem \ref{2.7}, we obtain the following.

\begin{ttheo}\label{last} Let $R$ be a Noetherian ring such that all non-trivial ideals of $R$ are vertices of $\Bbb{AG}(R)$ and $\gamma(\Bbb{AG}(R))<\infty$. If $R$ is a non-local ring, then $R$ is an Artinian ring with only finitely many ideals. Otherwise,  either $(R,\mathfrak{m})$ is a Gorenstein ring such that $\mathfrak{m}^n\neq \mathfrak{m}^{n+1}$ (for every positive integer $n$), $\Bbb{AG}(R)$ is a subgraph of Figure 1, and $\gamma(\Bbb{AG}(R))=0$ or $R$  is an local Artinian ring with maximal ideal $\mathfrak{m}$ such that $\mathfrak{m}^{t}\neq (0)$ and $\mathfrak{m}^{t+1}=(0)$. We have the following:

\noindent (a) If  $t=1$, then $R$ is either finite or a special principal ideal ring.

\noindent (b) If  $t=2$, then one of the following holds:

\noindent (b.1) $R$ is finite;

\noindent (b.2) $R$ is a special principal ideal ring;

\noindent (b.3)  ${\rm v.dim}_{R/{\mathfrak{m}}}{\mathfrak{m}}^{2}/{\mathfrak{m}}^{3}=1$, ${\rm v.dim}_{R/{\mathfrak{m}}}{\mathfrak{m}}/{\mathfrak{m}}^{2}=2$, $R$ has infinitely many ideals, $\Bbb{AG}(R)$ is a subgraph of Figure~2, and $\gamma(\Bbb{AG}(R))=0$.

\noindent (c) If $t=3$, then one of the following holds:

\noindent (c.1) $R$ is finite;

\noindent (c.2)   $R$ is a special principal ideal ring;

\noindent (c.3)  ${\rm v.dim}_{R/{\mathfrak{m}}}{\mathfrak{m}}/{\mathfrak{m}}^{2}=2$, ${\rm v.dim}_{R/{\mathfrak{m}}}{\mathfrak{m}}^{2}/{\mathfrak{m}}^{3}={\rm v.dim}_{R/{\mathfrak{m}}}{\mathfrak{m}}^{3}/{\mathfrak{m}}^{4}=1$, $R$ has infinitely many ideals, $\Bbb{AG}(R)$ is a subgraph of Figure 3, and  $\gamma(\Bbb{AG}(R))=0$.

\noindent (d) If $t\geq 4$, then $R$ is either finite or a special principal ideal ring.
\end{ttheo}

\begin{pproof}

If $R$ is a non-local by \cite[Theorem 3.5]{AB2012}, $R$ is an Artinian ring with only finitely many ideals. So, we assume that $R$ is a local ring.
As in the proof of Proposition \ref{nlocal}, we conclude that either  either $(R,\mathfrak{m})$ is a Gorenstein ring such that $\mathfrak{m}^n\neq \mathfrak{m}^{n+1}$ (for every positive integer $n$), $\Bbb{AG}(R)$ is a subgraph of Figure 1, and $\gamma(\Bbb{AG}(R))=0$ or $R$  is an Artinian ring. So assume that  $R$ is a local Artinian ring with maximal ideal $\mathfrak{m}$. Then there exists positive integer $t$ such that $\mathfrak{m}^{t}\neq (0)$ and $\mathfrak{m}^{t+1}=(0)$. If $|R/\mathfrak{m}|<\infty$, then one can easily check that $R$ is finite. Now, we may assume that $|R/\mathfrak{m}|=\infty$. We have the following cases according to the value of $t$:

{\bf Case 1:} $t=1$, i.e., $\mathfrak{m}^{2}=(0)$. Then by Case $1$ in Theorem  \ref{2.7}, $R$ is a
special principal ideal ring.

{\bf Case 2:} $t=2$, i.e., $\mathfrak{m}^{3}=(0)$. By Case 2 in Theorem   \ref{2.7}, either  $R$ is a
special principal ideal ring or ${\rm v.dim}_{R/{\mathfrak{m}}}{\mathfrak{m}}^{2}/{\mathfrak{m}}^{3}=1$, ${\rm v.dim}_{R/{\mathfrak{m}}}{\mathfrak{m}}/{\mathfrak{m}}^{2}=2$, $R$ has infinitely many ideals, $\Bbb{AG}(R)$ is a subgraph of Figure 2, and so $\gamma(\Bbb{AG}(R))=0$.

{\bf Case 3:} $t=3$, i.e., $\mathfrak{m}^{4}=(0)$. By Case 3 in Theorem   \ref{2.7}, either  $R$ is a
special principal ideal ring or  ${\rm v.dim}_{R/{\mathfrak{m}}}{\mathfrak{m}}/{\mathfrak{m}}^{2}=2$, ${\rm v.dim}_{R/{\mathfrak{m}}}{\mathfrak{m}}^{2}/{\mathfrak{m}}^{3}={\rm v.dim}_{R/{\mathfrak{m}}}{\mathfrak{m}}^{3}/{\mathfrak{m}}^{4}=1$, $R$ has infinitely many ideals, $\Bbb{AG}(R)$ is a subgraph of Figure 3, and so $\gamma(\Bbb{AG}(R))=0$.

 {\bf Case 4:} $t\geq 4$. By Case 4 in Theorem   \ref{2.7},  $R$ is a
special principal ideal ring.\hfill $\square$

\end{pproof}

Farid Aliniaeifard, Department of Mathematics, Brock University, St. Catharines, Ontario, Canada L2S 3A1. Tel: (905) 688-5550 and Fax:(905) 378-5713. E-mail address: fa11da@brocku.ca
\\

Mahmood Behboodi, Department of Mathematical of Sciences, Isfahan University of Technology, Isfahan, Iran 84156-8311. Tel : (+98)(311) 391-3612 and Fax : (+98)(311) 391-3602. E-mail address: mbehbood@cc.iut.ac.ir
\\

Yuanlin Li, Department of Mathematics, Brock University, St. Catharines, Ontario, Canada L2S 3A1. Tel: (905) 688-5550 ext. 4626  and Fax:(905) 378-5713. E-mail address: yli@brocku.ca

\end{document}